\documentclass{amsart}
\usepackage{amsmath,amssymb, amsthm, tikz}
\usepackage{pgfkeys}
\usepackage{graphics}
\usepackage{enumitem}

\newtheorem{theorem}{Theorem}[section]
\newtheorem{proposition}{Proposition}[section]
\newtheorem{lemma}{Lemma}[section]
\newtheorem{cor}{Corollary}[section]

\newenvironment{defi}{\medskip\noindent{\sc
Definition}. }{\goodbreak\medskip}

\newenvironment{nota}{\medskip\noindent{\sc
Notation}.}{\goodbreak\medskip}
\newenvironment{remk}{\noindent{\sc
Remark}. }{\goodbreak\vskip10pt}
\newenvironment{remks}{\noindent{\sc
Remarks}. }{\goodbreak\vskip10pt}
\newenvironment{notas}{\medskip\noindent{\sc
Notations}. }{\goodbreak\medskip}

\newenvironment{exa}{\noindent{\sc
Example}. }{\goodbreak\vskip10pt}
\newenvironment{ques}{\noindent{\sc
Question}. }{\goodbreak\vskip10pt}

\def\cal{\mathcal}

\def\P{{\mathcal P}}

\def\cb{{\mathcal B}}
\def\cF{{\mathcal F}}
\def\cs{{\mathcal S}}
\def\ct{{\mathcal T}}

\def\cC{{\mathcal C}}

\def\ci{{\mathcal I}}
\def\cc{{\mathcal C}}
\def\cg{{\mathcal G}}

\def\ca{{\mathcal A}}
\def\ch{{\mathcal H}}
\def\ck{{\mathcal K}}
\def\cl{{\mathcal L}}
\def\cm{{\mathcal M}}
\def\cN{{\mathcal N}}
\def\cR{{\mathcal R}}

\def\cv{{\mathcal V}}

\def\cb{\mathcal{B}}

\def\R{\mathbb{R}}
\def\Q{\mathbb{Q}}
\def\A{\mathbb{A}}
\def\Z{\mathbb{Z}}
\def\N{\mathbb{N}}

\def\T{\mathbb{T}}
\def\Q{\mathbb{Q}}

\def\smallskip{\par\vspace{1mm}}
\def\medskip{\par\vspace{2mm}}
\def\bigskip{\par\vspace{3mm}}

\newcount\equationcount
\def\thenumber{0}
\def\eq#1{\global\advance\equationcount by 1
   \def\thenumber{\number\equationcount}
                        {$$#1\eqno(\thenumber)$$}}

\tikzset{
xmin/.store in=\xmin, xmin/.default=-1.5, xmin=-1.5,
xmax/.store in=\xmax, xmax/.default=7.5, xmax=7.55,
ymin/.store in=\ymin, ymin/.default=-0.75, ymin=-0.75,
ymax/.store in=\ymax, ymax/.default=3.25, ymax=3.25,
}

\begin{document}

\title[Denjoy sub-systems ]{Denjoy sub-systems and Horseshoes}

\author{Marie-Claude Arnaud$^{\dag,\ddag,*}$ }

\email{Marie-Claude.Arnaud@imj-prg.fr}

\date{}



\thanks{$\dag$ Universit\'e de Paris,  CNRS, Institut de Math\'ematiques de Jussieu-Paris Rive Gauche\\ F-75013 Paris, France} 
\thanks{$\ddag$ member of the {\sl Institut universitaire de France.}}
\thanks{$*$This material is based upon work supported by the National Science Foundation under Grant No. DMS-1440140 while the author was in residence at the Mathematical Sciences Research Institute in Berkeley, California, during the [Fall] [2018] semester.}

\begin{abstract}  We introduce a notion of weak Denjoy sub-system (WDS) that generalizes the Aubry-Mather Cantor sets to diffeomorphisms of manifolds. We explain how a rotation number can be associated to such a WDS. Then we build in any horseshoe
 a continuous one parameter family of such WDS that is indexed by its rotation number. Looking at the inverse problem in the setting of Aubry-Mather theory, we also prove that for a generic conservative twist map of the annulus, the majority of the Aubry-Mather sets are contained in some horseshoe that is associated to a Aubry-Mather set with a rational rotation number.
  \end{abstract}

\maketitle

\section{Introduction and Main Results.}\label{SecIntro}

All the dynamicists know famous Poincar\'e sentence about periodic orbits.\\\

{\em Ce qui nous rend ces solutions p\'eriodiques si pr\'ecieuses, c'est qu'elles sont, pour ainsi dire, la seule br\`eche par o\`u nous puissions essayer de p\'en\'etrer dans une place jusqu'ici r\'eput\'ee inabordable.}\\

But  a periodic orbit for a  dynamical system $f:X\rightarrow X$   is simply a finite invariant subset and the Dynamics restricted to this set cannot be very complicated. What is more interesting is the Dynamics close to such a periodic orbit, that may give rise to various rich phenomena. For example, for a symplectic diffeomorphism of a surface, two kinds of  restricted Dynamics to invariant Cantor sets can exist close to the periodic orbits, that are
\begin{itemize}
\item horseshoes\footnote{We will define them precisely later in the article.} close to hyperbolic periodic points (see \cite{Sma1965}); since the work of Katok in \cite{Kat1980}, they are known to be the evidence of positive topological entropy. Moreover, they contain a dense set of periodic points;
\item aperiodic Aubry-Mather sets close to elliptic periodic points (see \cite{ALD1983}, \cite{Chen1985}, \cite{Mat1982}); they are known to have zero topological entropy and contain no periodic points. \end{itemize}
Although we will later focus on some specific horseshoes, we give here a general definition of horseshoe.

\begin{defi}\label{Defhs} Let $f:M\rightarrow M$ be a surface diffeomorphism. A {\em horseshoe for $f$} is a $f$-invariant subset $H\subset M$ 
such that the Dynamics $f_{|H}$ is $C^0$ conjugate to the one of a non-trivial transitive subshift with finite type. A horseshoe for $f$ is a {\em $\sigma_2$-horseshoe} when the Dynamics $f_{|H}$ is $C^0$ conjugate to the shift with two symbols.
\end{defi}
\begin{exa}
The first horseshoe was introduced by S.~Smale in \cite{Sma1965} close to a transversal homoclinic intersection of a hyperbolic periodic point. 
This horseshoe is hyperbolic. Burns and Weiss extended this in \cite{BurnWei1995} to  the case of topologically transversal homoclinic intersection. Le Calvez and Tal use purely topological horseshoes for 2-dimensional homeomorphisms in \cite{LeCTal2018}.
\end{exa}

The  category of aperiodic Aubry-Mather set was recently extended in \cite{ALC2017} to the notion of so-called {Denjoy sub-system} by P. Le Calvez and the author.  We recall the definition given in \cite{ALC2017}.
 
\begin{defi}
Let $f:M\rightarrow M$ be a $C^k$ diffeomorphism of a manifold $M$. A {\em $C^k$ (resp. Lipschitz)  Denjoy sub-system  for $f$} is a triplet $(K, \gamma, h)$ where
\begin{itemize}
\item $\gamma: \T\rightarrow M$ is a $C^k$ (resp. biLipschitz) embedding;
\item $h:\T\rightarrow \T$ is a Denjoy example with invariant compact minimal set $K\subset \T$;
\item $f(\gamma(K))=\gamma(K)$;
\item $\gamma\circ h_{\vert K}=f\circ\gamma_{\vert K}$.
\end{itemize}
\end{defi}

\begin{remks}
\begin{itemize}
\item In this definition, $\gamma(\T)$ is not necessarily invariant.
\item Observe the importance of $\gamma$ to fix the regularity of $\gamma(K)$. 
\item For $k=0$, what we call a $C^0$-diffeomorphism is in fact a homeomorphism and in this case we just require that $\gamma$ is a continuous embedding.
\item The  embedding is also  useful to define a circular order on the Cantor set $\gamma(K)$. 
\end{itemize}
\end{remks}

\begin{exa}
There exists different notions of Aubry-Mather sets for the exact symplectic twist maps of the annulus (see \cite{ALD1983} and \cite{Mat1982}). We will follow \cite{Ban1988}  and for us, an Aubry-Mather set is a totally ordered compact set that contains only minimizing orbits in a variational setting, see e.g. \cite{ALD1983}, \cite{Ban1988}. Let us recall some results that are contained in \cite{Ban1988} and \cite{Arna 2016} and that we will use. We fix an exact  symplectic twist map $f$ of the infinite annulus and a lift $F:\R^2\rightarrow \R^2$. Then
\begin{itemize}
\item every Aubry-Mather set is a partial Lipschitz graph;
\item every Aubry-Mather  set $\ca$ has a rotation number $\rho(\ca)\in\R$;
\item for every $r\in\R\backslash \Q$, there exists a unique maximal (for $\subset$) Aubry-Mather set $\ca_r$ with rotation number $r$ that contains every Aubry-Mather set with the same rotation number;
\item for every $r=\frac{p}{q}\in \Q$, there exist two  Aubry-Mather set $\ca_r^\pm$ with rotation number $r$ that  are maximal (for $\subset$) among the Aubry-Mather sets with the same rotation number.  They are are such that: $\forall x\in \tilde\ca_r^+, \pi_1\circ F^q(x)\geq \pi_1(x)+p$ (resp. $\forall x\in \tilde\ca_r^-, \pi_1\circ F^q(x)\leq \pi_1(x)+p$) where $\pi_1:\R^2\rightarrow \R$ is the first projection;
\item if $(\ca_n)$ is a sequence of Aubry-Mather sets such that the sequence of rotation numbers $(\rho(\ca_n))$ converges to some $r\in\R$, 
then $\bigcup_{n\in\N}\ca_n$ is relatively compact and any limit point of $(\ca_n)$ for the Hausdorff distance is an Aubry-Mather set with rotation number $r$.
\end{itemize}
The Aubry-Mather sets $\ca_r$ that have an irrational rotation number and that are not a complete graphs always contain a Lipschitz Denjoy sub-system $\cc_r$.
\end{exa}
We noticed that an important advantage of $\gamma$ is to define a circular order along $\gamma(K)$. But to do that, we only need the embedding restricted to $K$. That is why we introduce now a new notion, the one of {\em weak Denjoy sub-system} that extends the one of Denjoy sub-system. This notion is similar to the one of Denjoy set that was introduced by J.~Mather in \cite{Mat1985}.

\begin{defi}
Let $f:M\rightarrow M$ be a homeomorphism of a manifold $M$. A {\em weak Denjoy sub-system  for $f$} (in short WDS) is a triplet $(K, j, h)$ where
\begin{itemize}

\item $h:\T\rightarrow \T$ is a Denjoy example with invariant minimal set $K\subset \T$;
\item $j: K\rightarrow M$ is a homeomorphism onto its image;
\item $f(j(K))=j(K)$;
\item $j\circ h_{\vert K}=f\circ j$.
\end{itemize}
When $j$ is biLipschitz or a $C^k$ embedding (in the Whitney sense), we speak of Lipschitz or $C^k$ weak Denjoy sub-system  for $f$.\\
Two WDS $(K_1, j_1, h_1)$ and $(K_2, j_2, h_2)$ are {\em equivalent} if $j_1(K_1)=j_2(K_2)$.
\end{defi}

The restriction of a Denjoy sub-system to its non-wandering  set is always a WDS. On a surface, we have the reverse implication.

\begin{proposition}\label{TweakstrongDenjoy}
Let $(K, j, h)$ be a WDS of a surface homeomorphism. Then there exists a $C^0$ Denjoy sub-system $(K, \gamma, h)$ such that $\gamma_{|K}=j$.
\end{proposition}
\begin{remks}
\begin{itemize}
\item This result is specific to the case of surfaces because it uses a classical result on extension of homeomorphisms between Cantor sets of surfaces,.
\item We don't know about such a result with more regularity: Lipschitz, $C^1$ (a kind of Whitney extension theorem for diffeomorphisms). Observe that we proved in \cite{ALC2017} that there exists no $C^2$ Denjoy sub-system.
\end{itemize}
\end{remks}

 \begin{remk}
Let us recall that a  {\em circular order relation} on a set $X$ is a relation $\prec$ that is defined on the triplets of points of $X$ such that
\begin{itemize}
\item if $x, y, z\in X$, we have  $x\prec y\prec z$ or $z\prec y\prec x$; we use the notation $[x, z]_\prec=\{ y\in X; x\prec y\prec z\}$;
\item the two previous lines of inequalities are simultaneously satisfied if and only if $x=y$ or $y=z$;
\item if $x\prec y\prec z$, then $y\prec z\prec x$;
\item if $x\prec y\prec z$ and  $x\prec z\prec t$ then $x\prec y\prec t$.
\end{itemize}
If $\prec$ is a circular order on $X$, the inverse order $-\prec$ is defined by
$$\forall x, y, z\in X, x(-\prec) y(-\prec) z\Leftrightarrow z\prec y\prec x.$$
 \end{remk}
 \begin{notas} \begin{itemize}
 \item If $(K, j, h)$ is a WDS, we denote by $\prec_K$ the circular order on $j(K)$ that is deduced from the one of $K\subset \T$ via the map $j$.
\item The {\em graph $\cg(\prec_K)$} of this order relation is the set of the triplets $(a, b, c)\in(j(K))^3$ such that $a$, $b$ and $c$ are in this order along $j(K)$. This graph $\cg(\prec_K)$ is then a closed part of $(j(K))^3$ and then of $(M)^3$.\\
Observe that for every $a, c\in j(K)$, $\cg(\prec_K, a, c)=\{ b\in j(K); (a, b, c)\in \cg(\prec_K)\}$ is a non-empty compact subset of $M$, called an {\em interval} of $\cg(\prec_K)$. \end{itemize}
 \end{notas}
 
 \begin{remk}
 We have $\cg(\prec_K, a, a)=j(K)$ and for $a\not=c$, $\cg(\prec_K, a, c)$ contains at least $a$ and $c$. Moreover, we have $\cg(\prec_K, a, c)=\{ a, c\}$ if and only if $\{ a, c\}$ is one gap\footnote{Observe that in this case, $a$ and $c$ are $\alpha$ and $\omega$-asymptotic under the Dynamics.} of the Cantor est.

 \end{remk}

The first theorem we will prove allows us to extend Poincar\'e's notion of rotation number to WDS, or more precisely to the classes of equivalence of WDS.

\begin{theorem}\label{Trotasubsytem} Let $(K_1, j_1, h_1)$ and $(K_2, j_2, h_2)$ be two equivalent WDS for a same homeomorphism  $f:M\rightarrow M$ of a manifold $M$. Then
\begin{itemize}
\item there exists a homeomorphism $h:\T\rightarrow \T$ such that $h\circ h_1=h_2\circ h$.
\item 
we have $\prec_{K_1}=\prec_{K_2}$ or $\prec_{K_1}=-\prec_{K_2}$, hence the two orders have the same intervals.
\end{itemize} 


\end{theorem}

\begin{cor}\label{Crotasubsytem}
 The  map $\rho$ defined on the set of WDS with values in $\T/{x\sim -x}$ that associates to any WDS $(K, \gamma, h)$ the rotation number of $h$ modulo its sign is such that if $(K_1, \gamma_1, h_1)$ and $(K_2, \gamma_2, h_2)$ are equivalent, then $\rho(K_1, \gamma_1, h_1)=\rho(K_2, \gamma_2, h_2)$. 
\end{cor}


Let us endow the set of WDS with a topology that focus on their order relation.

\begin{defi}
Let $f:M\rightarrow M$ be a homeomorphism of a manifold $M$. Let $(K, j, h)$ be a weak Denjoy sub-system for $f$. Another weak Denjoy sub-system $(K_1, j_1, h_1)$ is close to $(K, j, h)$ if
\begin{itemize}
\item $j_1(K_1)$ and $j(K)$ are close to each other for the Hausdorff distance on compact subsets of $M$;
\item $\prec_{K_1}$ or $-\prec_{K_1}$ is close to $\prec_K$ in the following sense. The graph $\cg (\prec_K)$ and the graph $\cg(\prec_{K_1})$ (or $\cg(-\prec_{K_1})$) are close from each other for the Hausdorff topology on the compact subsets of $(M)^3$.
\end{itemize}
\end{defi}
 
\begin{proposition}\label{Prota}
The map that associates to every WDS its rotation number is continuous.
\end{proposition}

\begin{remk}
The previous result extends a result that is well-known in the setting of well-ordered sets for twist maps.
\end{remk}

Horseshoes and WDS are different but in general, it is believed that, up to some entropy restriction, horseshoes Dynamics contain every Dynamics (via symbolic Dynamics)\footnote{This is not completely correct because, for example, the dynamics of an odomoter cannot be embedded in a horseshoe even if it has zero entropy~: it is an isometry and the horseshoe is expansive.}.  

We will prove that for every horseshoe contains many WDS, and even a continuous 1-parameter family $(D_\rho)$ continuously depending on its rotationn number $\rho$ where $\rho$ is in a non-reduced to a point interval of $\T/x\sim -x$ of irrational numbers. 

\begin{theorem}\label{WDSSinHS} Let $f:M\rightarrow M$ be a $C^k$ diffeomorphism and let $\ch$ be a horseshoe for $f$. Then there is a $N\ge 1$ and a continuous map $D: r\in (\T\backslash\Q)/{x\sim -x}\mapsto (K_r, j_r, h_r)$ such that 
\begin{itemize}
\item $D(r)=(K_r, j_r, h_r)$ is a continuous WDS with rotation number $r$ for $f^N$;
\item $j_r(K_r)\subset \ch$.
\end{itemize}
Moreover, if $\ch$ is a $\sigma_2$-horseshoe, we have $N=1$.
\end{theorem}

Different authors before us built embedding of Denjoy Dynamics into horseshoes. For example, in \cite{HokHol1986}, the authors build a uncountable family of Denjoy Dynamics in a given horseshoe. If we analyse their construction, for every irrational rotation number, they build  a uncountable family of weak Denjoy sub-systems with two holes that are not conjugate together (see Markley, \cite{Mar1970},  for a characterization of conjugated Denjoy examples). In \cite{Bo2000}, looking for special invariant measures of the angle doubling on the circle, Bousch uses the one side shift on $\{0, 1\}^\N$ and the unique invariant measure with support in a  Cantor set analogous to the one we build.

\begin{remks}\begin{itemize}
\item The continuous WDS that we will embed in the horseshoe are WDS that have only a pair of orbits that are $\omega$ asymptotic (and then $\alpha$-asymptotic because we have a Denjoy Dynamics), i.e. that correspond to a Denjoy example with exactly one orbit of a wandering interval (we will say one gap).
\item Observe that the shift Dynamics is expansive. Hence we cannot embed in it a WDS with a infinite countable number of gaps: one of these gaps would have all its orbit with diameter less than the expansivity constant, which is impossible.
\item But it is possible to embed a family of WDS with a finite number $p$ of gaps in a $\sigma_{\sup\{ 2,p\}}$-horseshoe.
\item A similar method to embed Cantor sets with an interval of rotation numbers was proposed by K.~Hockett and P.~Holmes for dissipative twist maps in \cite{HokHol1986}. Here we proved a more general statement (for WDS) and also prove a continuity result.
 \end{itemize}
\end{remks}
\begin{cor}\label{DSSinHS} Let $f:M^{(2)}\rightarrow M^{(2)}$ be a $C^k$ diffeomorphism of a surface and let $\ch$ be a horseshoe for $f$. Then there is a $N\ge 1$ and a continuous map $D: r\in (\T\backslash\Q)/{x\sim -x}\mapsto (K_r, \gamma_r, h_r)$ such that 
\begin{itemize}
\item $D(r)=(K_r, \gamma_r, h_r)$ is a continuous  Denjoy sub-system with rotation number $r$ for $f^N$;
\item $\gamma_r(K_r)\subset \ch$.
\end{itemize}
Moreover, if $\ch$ is a $\sigma_2$-horseshoe, we have $N=1$.

\end{cor}
\begin{remks}
\begin{itemize}
\item We don't know how to choose continuously the embedding $\gamma_r$ or at least its image. But we don't need that to describe the Dynamics on $\gamma_r(K_r)$.\\
\item In the setting of the Aubry-Mather theory,  the map that associates to any Aubry-Mather set $\ca_r$ the graph that  linearly interpolates $\ca_r$ is in fact continuous when we endow the set of functions with the $C^0$ distance, and we will see that the Aubry-Mather sets with an irrational rotation number are actually contained in some horseshoes in the generic case. But the Aubry-Mather set don't continusously depend on the rotation number $r\in\R\backslash \Q$, so even in the case of the Aubry-Mather we don't kow if we can interpolate in a continuous way by a curve.
\end{itemize}
\end{remks}
We now focus on Aubry-Mather theory  and address the inverse problem: are the WDS that appear in a  natural way in symplectic 2-dimensional dynamics   contained in some horseshoe?\\

\begin{theorem}\label{TAMhorseshoe}
Let $f:\T\times \R\rightarrow \T\times\R$ be an exact  symplectic twist map and let $F:\R^2\rightarrow \R^2$ be a lift of $f$. Assume that $\ca_r^+$ (resp. $\ca_r^-$) is uniformly hyperbolic for some rational number $r\in\Q$.  Let $\cv_r$ be a neighbourhood of $\ca_r^+$ (resp. $\ca_r^-$). Then there exists a horseshoe $\ch_r^+$ (resp. $\ch_r^-$) for some $f^N$ and $\varepsilon >0$ such that
\begin{itemize}
\item $\ch_r^+$ (resp. $\ch_r^-$)  contains $\ca_r^+$ (resp. $\ca_r^-$) and is contained in $\cv_r$;
\item every Aubry-Mather set with rotation number in $(r, r+\varepsilon)$ (resp. $(r-\varepsilon, r)$) is contained in $\ch_r^+$ (resp. $\ch_r^-$);
\item every point in $\ch_r^+$ (resp. $\ch_r^-$)  has no conjugate points, i.e. has its orbit that is locally minimizing.
\end{itemize}
\end{theorem}

\begin{remks}
\begin{itemize}
\item It is well known that the topological entropy of a twist map restricted to the union of all its hyperbolic  Aubry-Mather sets is zero and has zero Hausdorff dimension (see \cite{Fa1989}). Theorem \ref{TAMhorseshoe} implies that for an open and dense subset of conservative twist diffeomorphisms (in any reasonable topology), there exists an invariant set $\ck$ of points with no conjugate points such that the Dynamics restricted to $\ck$ has positive topological entropy and positive Hausdorff dimension.
\item In \cite{LeC1987}, a transitive set that contains all the Aubry-Mather sets is built by P. Le Calvez. But this set is very different from the one we build here, because it contains in general orbits with conjugate points and is far from every Aubry-Mather set $\ca_r^\pm$. Moreover, it is not a horseshoe.
\item Observe that no WDS $(K, j, h)$ that is contained in a hyperbolic horseshoe is $C^1$. Indeed, the endpoints $a$, $b$ of every gap are $\alpha$ and $\omega$-asymptotic and their orbits are dense in $j(K)$. But for $n$ large enough, $f^n(a)$ and $f^n(b)$ are in the same local stable manifold and then oriented in the stable direction and $f^{-n}(a)$ and $f^{-n}(b)$ are in the same local unstable manifold and the oriented in the unstable direction. Hence, close to any point in $j(K)$, we find points such that the geodesic that joins them is  either along the stable or the unstable direction. So $j$ cannot be $C^1$. In the Aubry-Mather setting, it is Lipschitz.
\item As we noticed before, a weak Denjoy sub-system $(K, j, h)$ that is contained in a horseshoe has a finite number of gaps. When the horseshoe is uniformly hyperbolic with an expansivity constant equal to $\varepsilon$ and $j$ is $k$-biLipschitz, it can be proved that the number of gaps is at most $\frac{k}{\varepsilon}$. 
\end{itemize}
\end{remks}

A remarkable result of P.~Le ~Calvez asserts that general Aubry-Mather sets of general symplectic twist diffeomorphisms are uniformly hyperbolic (see\cite{LeC1988}). Joint with Theorem \ref{TAMhorseshoe}, this implies the following corollary.
\begin{cor}\label{CAMhorseshoe}
There exists a dense $G_\delta$ subset $\cg$ of the set of $C^k$ symplectic twist diffeomorphisms (for $k\geq 1$) such that for every $f\in \cg$, there exist an open and dense subset $U(f)$ of $\R$ and a sequence $(r_n)_{n\in\N}$ in $U(f)\cap \Q$ such that every minimizing Aubry-Mather set with rotation number in $U(f)$ is hyperbolic and contained in a horseshoe associated to a minimizing hyperbolic Aubry-Mather set whose  rotation number is $r_n$.
\end{cor}
\begin{remk}
Observe that in \cite{Go1985} , Goroff gives an example where the union of all the Aubry-Mather sets is uniformly hyperbolic.
\end{remk}
An open problem is the possible extension of Theorem \ref{TAMhorseshoe} in a relaxed setting. Hence we rise the following questions.

\begin{ques} (A.~Fathi)
Without assuming hyperbolicity, are the Aubry-Mather sets that are Cantor contained in some (non-hyperbolic) horseshoe?
\end{ques}

Another question concerns the Dynamics that are not necessarily twist diffeomorphisms.

\begin{ques} 
For a (eventually generic) symplectic diffeomorphism, is any  WDS contained in some horseshoe?
\end{ques}

It is possible to build $C^1$ or $C^2$ examples that have WDS that are not contained in horseshoes (examples that have a $C^1$ invariant curve on which the Dynamics is Denjoy, see \cite{Her1983}), but our question concerns higher differentiability.

\subsection*{Acknowledgements} The author is grateful to Fran\c cois B\'eguin, Pierre Berger, Sylvain Crovisier, Albert Fathi, Anna Florio, Patrice Le Calvez, Samuel Petite, Barbara Schapira, Jean-Paul Thouvenot and Maxime Zavidovique for insightful discussions that helped clarify several ideas.

\subsection{Notations}
For any hyperbolic periodic point $x$ of a $C^1$ diffeomorphism, we denote by $W^s(x,f)$ or $W^s(x)$ (resp. $W^u(x,f)$ or $W^u(x)$) its stable (resp. unstable) submanifold and by $W^s_{\rm loc}(x,f)$ or $W^s_{\rm loc}(x)$ (resp. $W^u_{\rm loc}(x,f)$ or $W^u_{\rm loc}(x)$) its local stable (resp. unstable) submanifold. We adopt exactly the same notations for non-necessarily periodic points that belong to some hyperbolic set.

Also we mention that the annulus is $\A=\T\times \R$, that is tangent space is $\A\times \R^2$ and that the tangent space at every point is endowed with its usual Euclidean norm. Moreover, we use the notation $\pi_1:\A\rightarrow \T$ for the first projection as well as its lift $\pi_1:\R^2\rightarrow \R$.

 \section{Proof of Proposition \ref{TweakstrongDenjoy}}

  We assume that $(K, j, h)$ is a weak Denjoy sub-system of a surface homeomorphism.  If we embed $\T$ in $\R^2$, then $K$ is a Cantor set that is subset of $\R^2$.  Moreover $j(K)$ is a Cantor subset of the surface $M$. We can find a covering of $j(K)$ by a finite number of open topological discs.  Decreasing them, we can assume that their closures are disjoint closed topological discs. Joining them, we find a closed annulus whose interior is an open annulus $\ca$ that contained $j(K)$. There exists a homeomorphism $\Phi: \ca\rightarrow \R^2\backslash \{(0, 0)\}$.  Then the Cantor subset $\Phi\circ j(K)$ of $\R^2$ is homeomorphic to the Cantor subset $K$ of $\R^2$. 
 We deduce from chapter 13 of \cite{Moi1977}  that there exists a homeomorphism $\psi:\R^2\rightarrow \R^2$ that extends the homeomorphism $j^{-1}\circ \Phi^{-1}:\Phi\circ j(K)\rightarrow K$.\footnote{Observe that this is specific to the 2-dimensional setting and that there exists some homeomorphisms between two Cantor subsets of $\R^3$ that cannot be extended to a homeomorphism of $\R^3$, see Theorem 5 of chapter 18 of \cite{Moi1977}.} Perturbing slightly $\psi$, we can assume that $(0, 0)\notin\psi(\T)$. Then $\gamma=\Phi^{-1}\circ \psi: \T\rightarrow \ca\subset M$ is a simple  continuous curve
  and $(K, \gamma, h)$ is a Denjoy sub-system that extends $(K, j, h)$.

\section{Proof of Theorem \ref{Trotasubsytem} and Corollary \ref{Crotasubsytem} }
\subsection{Proof of the first point of Theorem  \ref{Trotasubsytem} and of Corollary \ref{Crotasubsytem}}
We associate to any continuous dynamical system $F:X\rightarrow X$ an equivalence relation $\cR_F$ that is defined by
$$x\cR_F y\Leftrightarrow \lim_{k\rightarrow +\infty} d(F^kx, F^ky)=0.$$
Observe that if $H:X\rightarrow Y$ is a homeomorphism and if $X$ is compact, then we have
$$x\cR_Fy\Leftrightarrow H(x)\cR_{H\circ F\circ H^{-1}} H(y).$$
Hence $X/\cR_F$ is compact if and only if $Y/\cR_{H\circ F\circ H^{-1}} $ is compact. We denote by $p_F:X\rightarrow X/\cR_F$ the projection.

Because of Poincar\'e classification of circle homeomorphisms (see for example \cite{KatHass1995}), for every orientation preserving homeomorphism $h$ of the circle with an irrational rotation number, we have $\cR_h=\cR_{h^{-1}}$. Moreover, for such an orientation preserving homeomorphism of the circle with irrational rotation number, the relation is closed and it is also true for the restriction to any invariant compact subset. In this case, the quotient space, that corresponds to a closed equivalence relation on a compact space, is also compact. We then consider a semi-conjugation $k$ between the orientation preserving homeomorphism $h$ of the circle with irrational rotation number and a rotation $R_\alpha$, i.e. $k$ is non-decreasing continuous map onto the circle such that
$$k\circ h=R_\alpha \circ k.$$
Then $k:\T\rightarrow \T$ is continuous and we have
$$\forall x, y\in \T, k(x)=k(y)\Leftrightarrow x\cR_h y.$$
We denote by $K_h$ the unique non-empty minimal $h$-invariant compact subset (then $K_h$ is $\T$ or a Cantor subset) and we denote by $G_h$ the set of points of $K_h$ that are $\cR_h$ related to another point of $\T$. In other words, $G_h$ is the union of the endpoints of the gaps of the  set $K_h$.\\
Then there exists a unique map $\bar k:K_h/\cR_h\rightarrow \T$ such that   $\bar k\circ p_h=k$. The definition of the quotient topology implies that $\bar k$ is continuous and it is then a homeomorphism from $K_h/\cR_h$ to $\T$. Moreover, there exists a unique map $\bar h: K_h/\cR_h\rightarrow K_h/\cR_h$ that is the quotient Dynamics and that satisfies
$$\bar h\circ p_h=p_h\circ h;$$
we have then 
$$\bar k\circ \bar h= R_\alpha\circ \bar k,$$
i.e. $\bar k$ is a conjugation between $\bar h$ and $R_\alpha$.

Let us consider  two WDS $(K_1, j_1, h_1)$ and $(K_2, j_2, h_2)$ for a same homeomorphism  $f:M\rightarrow M$ of a manifold $M$ such that $C=j_1(K_1)=j_2(K_2)$. Let $k_i$ be a semi-conjugation between $h_i$ and a rotation $R_{a_i}$, i.e.
$$k_i\circ h_i=R_{a_i}\circ k_i.$$
As $f_{|C}=  j_i   \circ h_i\circ j_{i}^{-1}$, then  $C/\cR_f$ is homeomorphic to $K_i/\cR_{h_i}$ and so to $\T$. We denote by $p: C\rightarrow C/\cR_f$ the projection and by $\bar f:C/\cR_f\rightarrow C/\cR_f$ the reduced Dynamics. \\
\begin{center}
\includegraphics[width=9cm]{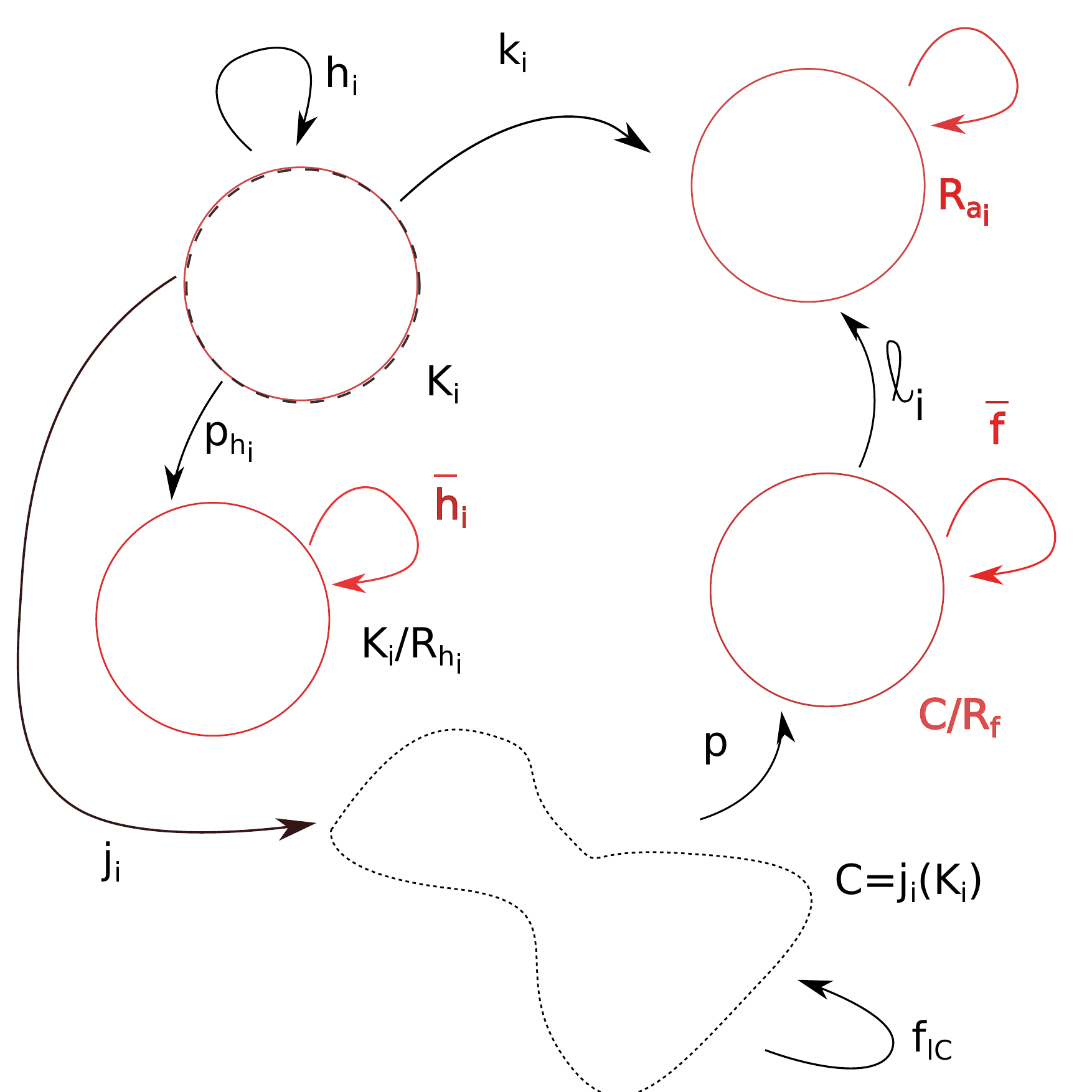}
\end{center}

Then the map $k_i\circ j_{i}^{-1}:C\rightarrow \T$ is a continuous surjection such that $$k_i\circ j_{i}^{-1}(x)=k_i\circ j_{i}^{-1}(y)\Leftrightarrow p(x)=p(y).$$
Hence, there exists a unique homeomorphism $\ell_i:C/\cR_f\rightarrow \T$ such   $\ell_i\circ p=k_i\circ j_{i}^{-1}$. We have then for all $y=p(x)\in C/\cR_f$
$$R_{a_i}\circ \ell_i(y)=R_{a_i}\circ k_i\circ j_{i}^{-1}(x)=k_i\circ h_i\circ j_{i}^{-1}(x)=k_i\circ j_{i}^{-1}\circ (j_i\circ h_i\circ j_{i}^{-1})(x)=k_i \circ j_{i}^{-1}\circ    f(x)=\ell_i(\bar f(y)).
$$

We deduce that
$$ R_{a_1}=\ell_1\circ \bar f\circ \ell_1^{-1}=(\ell_1\circ \ell_2^{-1})\circ R_{a_2}\circ(\ell_1\circ \ell_2^{-1})^{-1}.$$
As $R_{a_1}$ and $R_{a_2}$ are conjugate, we have $a_1=\pm a_2$. More precisely, $a_1=a_2$ when the conjugation preserves the orientation (and then is $(x\mapsto x+C)$) and $a_1=-a_2$ when the conjugation reverses the orientation (and then is $(x\mapsto C-x)$).  This gives Corollary \ref{Crotasubsytem} but doesn't end the proof of the first point of Theorem \ref{Trotasubsytem}. 

To finish the proof of this point, let us observe that 
$$j_1(G_{h_1}\cap K_1)=j_2(G_{h_2}\cap K_2)=\{ x\in C; \exists y\in C; y\not=x, y\cR_f x\}$$
is the set of the endpoints of the gaps of $f_{|C}$ (gaps are pairs of points that are $\omega$-asymptotic). We denote this set by $C_0$.

Thus we have $k_i(G_{h_i})=k_i\circ j_i^{-1}(C_0)=\ell_i\circ p(C_0)$.
We deduce that $k_1(G_{h_1})=\ell_1\circ \ell_2^{-1}(k_2(G_{h_2}))$. As $\ell_1\circ \ell_2^{-1}$ is either a translation $x\mapsto x+C$ or a symmetry $x\mapsto C-x$, there exists $C\in \R$ such that either $k_1(G_{h_1})=C+k_2(G_{h_2})$ or $k_1(G_{h_1})=C-k_2(G_{h_2})$. In other words, the image by $k_1$ of the union of the gaps of $K_{h_1}$ is the image by a translation or a symmetry of the image by $k_2$ of the union of the gaps of $K_{h_2}$. As explained in \cite{Her1979} and \cite{Mar1970}, this is equivalent to the fact that $h_1$ and $h_2$ are conjugated.
\subsection{Proof of the second point of Theorem  \ref{Trotasubsytem}}
For the second point, we know that $j_i:K_i\rightarrow C$ defines the order $\prec_{K_i}$. If we identify points that are $\omega$-asymptotic, we obtain a reduced order relation $\overline{\prec}_{K_i}$ on $K_i/\cR_{h_i}$ and  $C/\cR_f$ and  $\bar j_i:K_i/\cR_{h_i}\rightarrow C/\cR_f $ is an order preserving homeomorphism. As there are only two possibles orientations on the circle, we deduce for the two reduced order relations on $C/\cR_f$ that either they are equal or they are reverse. To deduce the result for the non-reduced relation, we have just to note that there is only one way 
to define the closed order relation $\prec_{K_i}$ on $C$ whose reduced relation is $\overline{\prec}_{K_i}$.

  \section{Proof of Proposition \ref{Prota} }
  Let us begin by explaining some results on the symbolic Dynamics of WDS. If $(K, j, h)$ is a WDS for $f$, we can encode the Dynamics in  the following non injective way\footnote{Observe that this is not necessarily the encoding that is given by the subshift of finite type on the horseshoe when this WDS is contained in some horseshoe.}. Let $x_0\in j(K)$ be a point of $j(K)$. We consider the interval $I_0$ of $j(K)$ of the points $y\in j(K)$ such that $x_0$, $y$ and $f(x_0)$ are in this order for $\prec_K$. We decide that $x_0\in I_0$ but $f(x_0)\notin I_0$. We denote by $I_1=j(K)\backslash I_0$ the complement of $I_0$ in $j(K)$. Then we consider the map that associates to every point $x\in j(K)$ its {\em itinerary}
  $$\ci(x)=(n_k(x))_{k\in\Z}$$
  where $f^k(x)\in I_{n_k(x)}$. When $x_0$ is the right end of a gap (a gap is the image by $j$ of the two endpoints of a wandering interval of $h$) of $j(K)$, $I_0$ and $I_1$ are closed and open in $j(K)$ and then $\ci$ is continuous\footnote{We will prove in section \ref{ssturm} that when $h$ is a Denjoy example with one gap, then $\ci$ is in fact a homeomorphism on its image.}.
  
 We assume that $x_0$ is indeed the right end of a gap of $j(K)$ and we denote by $\ck$ the set $\ci(K)$. As the Denjoy example is semi-conjugate to the rotation with angle $\alpha =\rho(h)$, $\ci(x_0)$ is nothing else than the Sturmian sequence that is associated to the rotation $R_\alpha$, i.e.  (see \cite{Fogg2002})  $n_k(x)=0$ if and only if $k\alpha\in [0, \alpha)$.   
  
  Let us now consider a WDS $(K_1, j_1, h_1)$ that is close to $(K, j, h)$ for the topology that we defined before. Let $(x_1, x_0)$ be the gap whose $x_0$ is the right end in $j(K)$. Then the interval $\cg(\prec_K, x_1, x_0)=\{x_0, x_1\}$ has only two points. As $\cg(\prec_{K_1})$ is close to $\cg(\prec_K)$ for the Hausdorff distance, there exists two points $y_1, y_0\in j(K_1)$ that are close to $x_1, x_0$ and such that $\cg(\prec_{K_1},y_1, y_0)$ is contained in a neighbourhood of $\cg(\prec_K, x_1, x_0)$. As we know that $y_1, y_0\in \cg(\prec_{K_1},y_1, y_0)$, that $y_0$ is close to $x_0$ and that $y_1$ is close to $x_1$, this implies that $\cg(\prec_{K_1},y_1, y_0)$ is close to $\cg(\prec_K, x_1, x_0)$ for the Hausdorff distance. Then we write $$\cg(\prec_{K_1},y_1, y_0)=\cg_0\cup \cg_1$$ where the points of $\cg_0$ are close to $x_0$ and the points of $\cg_1$ are close to $x_1$. Observe that $\cg(\prec_{K_1},y_1, y_0)$ is an interval for $\prec_{K_1}$, where $\prec_{K_1}$ define a (non circular) total order. Hence we can define $z_1=\sup \cg_0$ and $z_0=\inf\cg_1$. Then  $\{z_1, z_0\}$ is a gap of ${K_1}$  that is close to $ \{x_1, x_0\}$ for the Hausdorff topology.  We then associate to $z_1$ its itinerary exactly as we did for $x_1$. Let us fix $N\geq 1$. Then if $(K_1, j_1, h_1)$  is close enough to $(K, j, h)$, the two itineraries between $-N$ and $N$ match. But these itineraries determine the first terms of the continued fraction of the two rotations numbers of $h_1$, $h$ (see \cite{Fogg2002}). Because they coincide up to the order N, we deduce that $\rho(h_1)$ is close to $\rho(h)$ and then the continuity of the rotation number.

  \section{Proof of Theorem \ref{WDSSinHS} and Corollary \ref{DSSinHS}}\label{ssturm}
  \subsection{Proof of Theorem \ref{WDSSinHS}}We will use the following notions.
  
  \begin{defi}
A {\em $n$-cylinder} in $\Sigma_2$ is a set of sequences  $(u_k)_{k\in\Z}\in \{0, 1\}^\Z$such that  $u_{-n}=\delta_{-n}; \dots, u_0=\delta_0; \dots ;u_n=\delta_n$ where the $\delta_i$s are fixed in $\{ 0,1\}$.\\
As $d((u_k)_{k\in\Z}, (v_k)_{k\in \Z})=\max_{k\in\Z}\frac{|u_k- v_k|}{|k|+1}$, observe that a $n$-cylinder is exactly a closed ball with radius $\frac{1}{n+2}$.\\
A $n$-words of $u$ is a sequence of $n$ successive terms of $u$.
\end{defi}

  Let $f:M\rightarrow M$ be a $C^k$ diffeomorphism and let $\ch$ be a horseshoe for $f$. Then there exists a transitive subshift with finite type $\sigma_A: \ck \rightarrow \ck$ that is defined on some shift invariant compact subset $\ck$ of $\Sigma_p$ such that $f_{|\ch}$ is $C^0$ conjugate to $\sigma_A$. Then there exists a $\sigma_A$-invariant compact subset $\ck_0\subset \ck$ and $N\geq 1$ such that $\sigma^N_{A|\ck_0}$ is $C^0$ conjugate to $\sigma_2$. Hence we just need to prove the theorem for a $\sigma_2$-horseshoe to deduce the general statement.   We assume that $f_{|\ch}=k\circ \sigma_2\circ k^{-1}$.
  
  Let now $h_\alpha: \T\rightarrow \T$ be a Denjoy example with minimal Cantor set $C_\alpha$ such that
  \begin{itemize}
  \item $\T\backslash C_\alpha$ is the orbit of one interval $I_\alpha=(a_\alpha, b_\alpha)$;
  \item the rotation number of $h$ is $\alpha$.
  \end{itemize}
  We consider two disjoint segments $I_0(\alpha)$ and $I_1(\alpha)$ in $\T$ such that
  \begin{itemize}
  \item one endpoint of $I_j(\alpha)$ is in $I_\alpha$ and the other one is in $h_\alpha(I_\alpha)$.
  \item $I_0(\alpha)$ joins $I_\alpha$ to $h_\alpha(I_\alpha)$ in the direct sense.
  \end{itemize}
  Let $k_\alpha:\T\rightarrow \T$ be a semi-conjugation between $h_\alpha$ and $R_\alpha$, i.e. $k_\alpha\circ h_\alpha=R_\alpha\circ k_\alpha$. Then, the intervals $I_0(\alpha)$ and  $I_1(\alpha)$ are mapped on intervals $K_0=[0, \alpha]$ and $K_1=[\alpha, 1]$. As $\alpha$ is irrational, if $(n_k)_{k\in\Z}\in \Sigma_2$ is any sequence of $0$ and $1$, there exists at most one $\theta\in \T$ such that, for every $k\in\Z$, we have $\theta+k\alpha\in K_{n_k}$. \\
Let us now consider two points $\theta_1\not=\theta_2$ in $C_\alpha$ such that for every $k\in\Z$, $h_\alpha^k(\theta_1)$ and $h_\alpha^k(\theta_2)$ belong to a same interval $I_{n_k}(\alpha)$. Then for every $k\in \Z$, the points $k_\alpha\circ h_\alpha^k(\theta_1)=k_\alpha(\theta_1)+k\alpha$ and $k_\alpha\circ h_\alpha^k(\theta_2)=k_\alpha(\theta_2)+k\alpha$ belong to the same interval $K_{n_k}$ and so $k_\alpha(\theta_1)=k_\alpha(\theta_2)$, i.e. $\theta_1$ and $\theta_2$ are the two endpoints of some gap of the Cantor set $C_\alpha$. So there exists   $k\in\Z$ such that $h_\alpha^k(\theta_1)$ and $h_\alpha^k(\theta_2)$ are the two endpoints of $I_\alpha$ for example $I_\alpha=(h_\alpha^k(\theta_1), h_\alpha^k(\theta_2))$. But this implies that $h_\alpha^k(\theta_1)\in I_{1}(\alpha)$ and $h_\alpha^k(\theta_2)\in I_{0}(\alpha)$ and this contradicts that for every $k\in\Z$, $h_\alpha^k(\theta_1)$ and $h_\alpha^k(\theta_2)$ belong to a same interval $I_{n_k}(\alpha)$.  So we have proved that if we use the notation for $\theta\in C_\alpha$ that $h_\alpha^k(\theta)\in I_{n_k(\theta)}$, then the map $\ell_\alpha:C_\alpha\rightarrow \Sigma_2$ defined by $\ell_\alpha(\theta)=(n_k(\theta))_{k\in\Z}$ is injective. As the $I_k(\alpha)\cap C_\alpha$ are open (and closed) in $C_\alpha$, this map is also continuous and then is a homeomorphism onto its image.\\
This provides a homeomorphism from $C_\alpha$ onto $\ell_\alpha(C_\alpha)\subset \Sigma_2$ such that 
$$\forall \theta\in C_\alpha, \ell_\alpha\circ h_\alpha(\theta)=\sigma_2\circ \ell_\alpha(\theta).$$
The WDS with rotation number $\alpha\in [0, 1/2)\backslash \Q$ that we consider is then $(C_\alpha, j_\alpha=k\circ \ell_\alpha, h_\alpha)$. \\

Observe that $\ell_\alpha(b_\alpha)=(n_k(b_\alpha))_{k\in\Z}$ is the Sturmian sequence that is associated to the rotation $R_\alpha$. Let us recall that if  $u=(u_k)_{k\in\Z}$ is a Sturmian sequence, then for every $n\geq 1$, there are exactly $n+1$ $n$-words in $u$. 
  As $h_{\alpha|C_\alpha}$ is minimal, the orbit of $\ell_\alpha(b_\alpha)$ under $\sigma_2$ is dense in $\ell_\alpha(C_\alpha)$. Let us now fix $\alpha_0\in [0, 1/2)\backslash \Q$ and  $n\geq 1$. There exists $N\geq 1$ such that all the $m$-words in $\ell_{\alpha_0}(b_{\alpha_0})$ with $m\leq 2n+1$ are contained in $(n_k(b_{\alpha_0}))_{k\in[-N, N]}$. If now $\alpha$ is close enough to $\alpha_0$, $(n_k(b_{\alpha}))_{k\in[-N, N]}$ is equal to $(n_k(b_{\alpha_0}))_{k\in[-N, N]}$. As $\ell_\alpha(b_\alpha)=(n_k(b_\alpha))_{k\in\Z}$ is Sturmian, this implies that all the $m$-words in $\ell_{\alpha}(b_{\alpha})$ with $m\leq 2n+1$ are contained in $(n_k(b_{\alpha}))_{k\in[-N, N]}=(n_k(b_{\alpha_0}))_{k\in[-N, N]}$, which means that the distance between the $\sigma_2$ orbits of $\ell_\alpha(b_\alpha)$ and $\ell (b_{\alpha_0})$ is less than $\frac{1}{n+2}$. 
  This implies that $\ell_\alpha(C_\alpha)$ is $\frac{1}{n}$-close to $\ell_{\alpha_0}(C_{\alpha_0})$. Hence $j_\alpha(C_\alpha)=k(\ell_\alpha(C_\alpha))$ is close to $j_{\alpha_0}(C_{\alpha_0})=k(\ell_{\alpha_0}(C_{\alpha_0}))$.

We want now to prove that $\cg(\prec_{C_\alpha})$ is close to $\cg(\prec_{C_{\alpha_0}})$. In a equivalent way, we can work in $\Sigma_2$ instead of $\ch$ and assume that the graphs of $\cg(\prec_{C_\alpha})$ and $\cg(\prec_{C_{\alpha_0}})$ are in $(\Sigma_2)^3$. 
Then the intersection of the $n$ cylinder $C(\delta_{-n},\dots, \delta_0, \dots, \delta_n)=\{ (u_k)_{k\in\Z}; \forall k\in [-n, n], u_k=\delta_k\}$ with $\ell_\alpha(C_\alpha)$  is an interval for the order $\prec_{C_\alpha}$, that is before encoding the intersection of intervals $\bigcap_{k=-n}^{k=n}h_\alpha^{-k}(I_{\delta_k})$. This interval is non-empty if and only if $(\delta_i)_{i\in [-n, n]}$ is a $(2n+1)$-word of the Sturmian sequence $(n_k(b_{\alpha_0}))_{k\in\Z}$ for $\alpha_0$.\\
Let us now fix a $n\geq 1$. There exists $N\geq 1$ such that all the admissible $(2n+1)$-words of $(n_k(b_{\alpha_0}))_{k\in \Z}$ are contained in the sequence $(n_k(b_{\alpha_0}))_{k\in [-N, N]}$. There exists a neighbourhood $V$ of $\alpha_0$ in $\T$ such that, for every $\alpha\in V$, we have
\begin{itemize}
\item $\forall k\in[-N, N], n_k(b_\alpha)=n_k(b_{\alpha_0})$;
\item the intervals $C(n_{k-n}(b_\alpha),\dots, n_k(b_\alpha), \dots, n_{k+n}(b_\alpha))\cap \ell_\alpha(C_\alpha)$ and\\
 $C(n_{k-n}(b_{\alpha_0}),\dots, n_k(b_{\alpha_0}), \dots, n_{k+n}(b_{\alpha_0}))\cap \ell_\alpha(C_{\alpha_0})$ for $n-N\leq k\leq N-n$ (that are $\frac{1}{n}$-close to each other) are in the same order, for $\prec_{K_\alpha}$ for the first ones and for $\prec_{K_{\alpha_0}}$ for the second ones, because it is the order of this intervals for the two rotations.
\end{itemize}
We deduce that $\cg(\prec_{C_\alpha})$ is $\frac{1}{n}$-close to $\cg(\prec_{C_{\alpha_0}})$.

\subsection{Proof of Corollary \ref{DSSinHS}}
It is a corollary of Theorems \ref{TweakstrongDenjoy} and \ref{WDSSinHS}.

\section{Proof of Theorem \ref{TAMhorseshoe} and Corollary \ref{CAMhorseshoe}}
\subsection{Proof of Theorem \ref{TAMhorseshoe}}
 
 We assume that $f:\T\times \R\rightarrow \T\times\R$ is a symplectic twist map and that $F:\R^2\rightarrow \R^2$ is one of its lifts. We assume that   $\ca_r^+$  is uniformly hyperbolic for some rational number $r\in\Q$.   We want to prove that there exists a horseshoe $\ch_r^+$   for some $f^n$ and $\varepsilon >0$ such that
\begin{itemize}
\item $\ch_r^+$   contains $\ca_r^+$;
\item every Aubry-Mather set with rotation number in $(r, r+\varepsilon)$  is contained in $\ch_r^+$;
\item every point in $\ch_r^+$   has no conjugate points, i.e. has its orbit that is locally minimizing.
\end{itemize}
We write $r=\frac{p}{q}$ as an irreducible fraction. As $\ca_r^+$ is a compact uniformly hyperbolic set, it has a finite number of $q$-periodic points.  We denote them by $x_1, \dots, x_n$ in the usual cyclic order along  $\T$ (for the first projection). Then $\ca_r^+$ is the union of these periodic points and some heteroclinic orbits between these heteroclinic points (see e.g. \cite{Ban1988}). Moreover, such heteroclinic orbit for $f^q$ that is contained in $\ca_r^+$ can only connect an $x_k$ to $x_{k+1}$ (with $x_{n+1}=x_1$). \\
If two heteroclinic orbits in $\ca_r^+$ connect the same $x_k$ to $x_{k+1}$, then we can choose an order on these to orbits $(y_j)_{j\in\Z}$ and $(z_j)_{j\in\Z}$ for $f^q$ such that

$$x_k<\dots< y_{j-1}<z_{j-1}<y_j<z_j<y_{j+1}<z_{j+1}<\dots <x_{k+1}.$$
Let $\varepsilon>0$ be an expansivity constant for $f^q_{|\ca_r^+}$ and let $K$ be a Lipschitz constant of the Aubry-Mather set $\ca_r^+$ as a graph. Then for some $j$ we have $d(y_j, z_j)\geq \varepsilon$. Hence the distance between  the  first projection of $y_j$, $z_j$   is more than $\frac{\varepsilon}{1+K}$ for some $j$. Of course we can use the same argument for any finite set of heteroclinic orbits $(y^1_j)_{j\in\Z}$, \dots , $(y^N_j)_{j\in\Z}$ connecting $x_k$ to $x_{k+1}$ in $\ca_r^+$. We have
$$x_k\dots <y^1_{j-1}<\dots <y^N_{j-1}<y^1_j<\dots <y^N_j<\dots <x_{k+1},$$
and we find $N$ integers $j_1$, \dots, $j_N$ such that (with the convention $y^{N+1}_j=y^1_{j+1}$)
$$\forall i\in \{ 0, N\}, d(y_{j_i}^i, y_{j_i}^{i+1})\geq \varepsilon.$$
Then the intervals $(\pi_1(y^1_{j_1}), \pi_1(y^2_{j_1}))$, \dots, $(\pi_1(y^N_{j_N}), \pi_1(y^{N+1}_{j_N}))$ are disjoint intervals in $\T$ with length larger or equal to $\frac{\varepsilon}{K+1}$. This implies that $N\leq \frac{K+1}{\varepsilon}$. Hence $\ca_r^+$ is a hyperbolic set that is the union of periodic orbits and of a {\em finite} number of heteroclinic orbits. Moreover, there always exists at least a heteroclinic connection in $\ca_r^+$ between two adjacent periodic points in $\ca_r^+$ (see \cite{Ban1988}). Hence $\ca_r^+$ is a cycle of transverse heteroclinic intersections with  period $q$ (see definition in Appendix \ref{Appa}).

We introduce the notation $p:\R^2\rightarrow \T\times \R$ for the usual projection. When $E\subset \T\times \R$, we denote by $\tilde E=p^{-1}(E)$ its lift.

Let us fix a neighbourhood $\cN$ of $\ca_r^+$. Then $\ca_r^-  \backslash \cN$ is finite because   $\ca_r^- $ is the union of  $\ca_r^- \cap  \ca_r^+ $ and the   union of a finite number of orbits that are homoclinic to $\ca_r^- \cap  \ca_r^+ $.  For every $x\in \tilde\ca_r^-  \backslash \tilde\cN$ , we have $\pi_1\circ F^q(x)<\pi_1(x)+p$. Then $\varepsilon=\min\{ \pi_1(x)+p-\pi_1\circ F^q(x); x\in  \tilde\ca_r^-  \backslash \tilde\cN\}$ is a positive number. We introduce the open set
$${\cal U}=p\left(\{ x\in \R^2; \pi_1(x)+p-\pi_1\circ F^q(x)>\frac{\varepsilon}{2}\}\right)$$
that contains $\ca_r^-  \backslash \cN$.
Then $\cN\cup {\cal U}$ is a neighbourhood of $\ca_r^+\cup \ca_r^-$. As the rotation number is continuous and has the union of minimizing orbits is closed, there exists $\eta>0$ such that every Aubry-Mather set with rotation number in $(r-\eta, r+\eta)$ is in $\cN\cup{\cal U}$. If moreover  $\ca$ is an Aubry-Mather set with rotation number in $(r, r+\eta)$, then we have
$$\forall x \in \tilde\ca, \pi_1\circ F^q(x)>\pi_1(x)+p.$$
Hence $\ca\cap {\cal U}=\emptyset$ and  thus $\ca\subset {\cal N}$.\\
We have the proved that there  exists $\eta>0$ such that  every Aubry-Mather set with rotation number in $(r, r+\eta)$  is contained in $\cN$.

We then use subsection \ref{ssA2}  of Appendix \ref{Appa}. There exists $N\geq 1$ and a neighbourhood $\cN$ of the cycle of transverse heteroclinic intersections with  period $q$ $\ca_r^+$, such that the maximal $f^{qN}$ invariant set contained in $\cN$ is a horseshoe $\ch_r^+$ for $f^{qN}$ (see Definition \ref{Defhs}). This horseshoe then satisfies the two first points of Theorem \ref{TAMhorseshoe}.

Moreover, observe that along $\ca_r^+$, there exists a $Df$ invariant field of half-lines (the half Green bundles $g_+$ of $G_+$, see \cite{Arna 2016}) transverse to the  vertical fiber, that is a subset of the unstable bundle along $\ca_r^+$. By continuity of the unstable bundle along any hyperbolic set, we can extend $g_+$ to the whole $\ch_r^+$ into a field of half-line that are contained in the unstable bundle. If $\cN$ is small enough, this field as well as its first $qN$ images by $Df$ is also transverse to the vertical. This implies the last point of Theorem \ref{TAMhorseshoe}.

\subsection{Proof of Corollary \ref{CAMhorseshoe}}
We use the results of P.~Le~Calvez that are in \cite{LeC1988}. We consider the $G_\delta$ subset $\cg$ of the set of $C^k$ symplectic twist diffeomorphisms  $f$ whose elements satisfy the following conditions.
\begin{itemize}
\item if $x$ is a periodic point for $f$ with smallest period $q$, none of the eigenvalues of $Df^q(x)$ is a root of unity;
\item all the heteroclinic intersections between invariant manifolds of hyperbolic periodic points are transverse.
\end{itemize}
It is proved in  \cite{LeC1988} that all the Aubry-Mather sets that have a rational rotation number are then hyperbolic. By Theorem  \ref{TAMhorseshoe}, for every $r\in\Q$, there exists an open interval $(r-\varepsilon_r, r+\varepsilon_r)$ such that every Aubry-Mather set with rotation number in this interval is contained in the horseshoe $\ch_r^+$ or the horseshoe $\ch_r^-$. This gives the conclusion of the corollary for 
$$U(f)= \bigcup_{r\in\Q}(r-\varepsilon_r, r+\varepsilon_r).$$

 \appendix
 \section{On horseshoes}\label{Appa}
 In this section, we will be interested in some horseshoes that are related to the heteroclinic intersections. Generally, authors look at what happens  close to one homoclinic point associated to a periodic point (in \cite{BurnWei1995}, the authors also consider heteroclinic  connections for two fixed points). But to apply our results to Aubry-Mather sets, we will need to study the horseshoes that can be built by using a (circular) family of periodic points and heteroclinic intersections. Let us explain this now.

\subsection{Introduction to heteroclinic horseshoes}\label{ssA1} We will consider heteroclinic cycles. For a diffeomorphism $f:M\rightarrow M$ of a surface, we will call a $q$-periodic point $x$ a {\em saddle} if the two eigenvalues $\lambda$, $\mu$ of $Df^q(x)$ are positive and such that $\mu<1<\lambda$.

\begin{defi}
Let $f:M\rightarrow M$ be a surface diffeomorphism. A {\em cycle of transverse heteroclinic intersections with period 1} is determined by
\begin{itemize}
\item a finite cyclically ordered set of saddle hyperbolic fixed points $x_{n+1}=x_1, \dots , x_n$ with an orientation on each submanifold $W^s(x_i)$ and $W^u(x_i)$;
\item for every $k\in [1, n]$ a non-zero finite number $n_k$ of transverse heteroclinic orbits in $W^u(x_k,f)\cap W^s(x_{k+1},f)$: $x_k$ and  $y_1^k, \dots , y_{n_k}^k$ are in this order along $W^u(x_k,f)$ and $y_1^k, \dots , y_{n_k}^k$, $x_{k+1}$ also along $W^s(x_{k+1},f)$\footnote{This implies that the $y_i^k$ are all on a same branch of $W^u(x_k,f)$ and   $W^s(x_{k+1},f)$.}. Moreover, they define different orbits. 

\end{itemize}

\end{defi}
\begin{center}
\includegraphics[width=8cm]{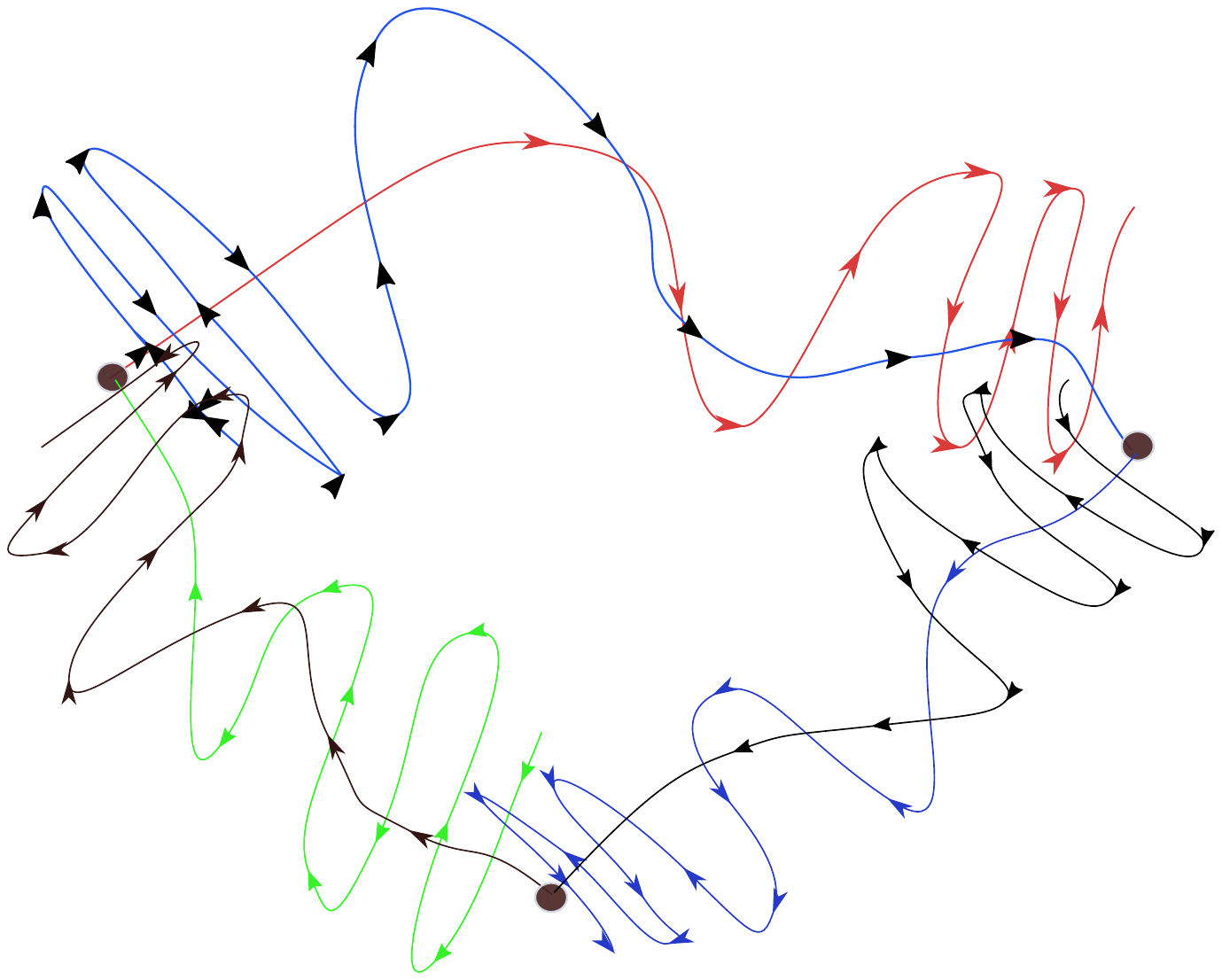}
\end{center}

\begin{defi}
Let $f:M\rightarrow M$ be a surface diffeomorphism and let $q\geq 1$ be an integer. A {\em cycle of transverse heteroclinic intersections with  period $q$} is determined by
\begin{itemize}
\item a finite cyclically ordered set of saddle hyperbolic $q$-periodic points $x_{nq+1}=x_1, \dots , x_{nq}$  such that this order is preserved by $f$  with an orientation on each submanifold $W^s(x_i)$ and $W^u(x_i)$; we assume that every set \\
$\{ x_i, x_{i+n}, \dots x_{i+(q-1)n}\}$ is an orbit;
\item for every $k\in [1, qn]$ a non-zero finite number $n_k$ of transverse heteroclinic orbits in $W^u(x_k,f)\cap W^s(x_{k+1},f)$: $x_k$ and  $y_1^k, \dots , y_{n_k}^k$ are in this order along $W^u(x_k,f)$ and $y_1^k, \dots , y_{n_k}^k$, $x_{k+1}$ also along $W^s(x_{k+1},f)$\footnote{This implies that the $y_i^k$ are all on a same branch of $W^u(x_k,f)$ and   $W^s(x_{k+1},f)$.}. Moreover, they define different orbits.
\item we also assume $n_{k+n}=n_k$, that $x_k$ and $x_{n+k}$ are on a same orbit and that $y_j^k$ and $y_j^{n+k}$ are on the same orbit.

\end{itemize}

\end{defi}

\begin{nota}
Now we consider a cycle of transverse heteroclinic intersections $\ch$ with period $q$ for $f$ that is given by the $x_ks$ and the $y_j^{k}s$ as before.  We denote by $K(\ch)$ the union of the orbits of the $x_k$s and the $y_j^k$s.
\end{nota}
\begin{remk}
Observe that $K(\ch)$ is a $f$-invariant  compact set that is uniformly hyperbolic.  We denote by $E$ the tangent bundle $TM$.  By \cite{Yoccoz1995}, we can translate the hyperbolicity condition by using some cones. This is an open condition and we can extend these cones to a compact neighbourhood $\cv$ of $K(\ch)$ such that
\begin{itemize}
\item there exists a continuous splitting $E=E^1\oplus E^2$ on $\cv$ that coincides with $E=E^s\oplus E^u$ on $K(\ch)$ and two norms $|\cdot|_i$ on $E^i$ such that $$C_x=\{ v=v_1+v_2, v_1\in E_x^1, v_2\in E_x^2, |v_1|_{1, x}\leq |v_2|_{2, x}\};$$
the family $(C_x)_{x\in \cv}$ is the associated cone field; the dual cone field is the family $(C^*_x)_{x\in \cv}$ defined by $C_x^*=E_x\backslash {\rm int}C_x$.;
\item for some constant $c>1$, we have for every $x\in \cv$, $v_1\in E^1_x$ and $v_2\in E^2_x$
$$c^{-1}\| v_1+v_2\|\leq \max\{ |v_1|_{1, x}, |v_2|_{2,x}\}\leq c\| v_1+v_2\|_x.$$
\item there exists an integer $m\geq 1$ and a constants $\mu>1$ so that
\begin{enumerate}
\item for $x\in \cv$, $Df(C_x)\subset \widetilde C_{\mu,f(x)}$ where
$$\widetilde C_{\lambda,x}=\{ v=v_1+v_2\in E_x; \mu|v_1|_{1, x}\leq |v_2|_{2, x}\};$$
\item for $x\in \cv$, for $v\in C_x$, $\| Df^m(v)\|_{f^m(x)}\geq \mu.\| v\|_x$;
\item for $x\in \cv$, for $v\in C^*_x$, $\| Df^{-m}(v)\|_{f^{-m}(x)}\geq \mu.\| v\|_x$.
\end{enumerate}
\end{itemize}
We  define 
$$\ck(\cv)=\bigcap_{k\in\Z}f^k(\cv).$$
Then $\ck(\cv)$ is compact and hyperbolic. Let $\varepsilon>0$ be  a constant of expansivity , i.e. such that
$$\forall x, y\in\ck(\cv), \left(\forall k\in \Z, d(f^kx, f^ky)<\varepsilon\right)\Rightarrow x=y.$$

Choosing eventually a smaller neighbourhood, we can assume that the diameter of every connected component of $\cv$ is smaller than $\varepsilon$, and also that $\cv$ has a finite number $N$ of connected components that all meet $K(\cv)$.  

We denote by $\cC_1, \dots , \cC_N$ the connected components of $\cv$ and define the itinerary function $H:\ck(\cv)\rightarrow \Sigma_N$ by $f^{k}(x)\in\cC_{H(x)_k}$. Hence the $k$th component of $H(x)$ corresponds to the connected component of $\cv$ that contains $f^kx$. Then $H$ is continuous.
Because of the expansiveness property, $H$ is injective, so $H$ is a homeomorphism from $\ck(\cv)$ onto $H(\ck(\cv))\subset\Sigma_N$ such that
$$\forall x\in \ck(\cv), \sigma\circ H(x)=H\circ f(x).$$
\end{remk}
But in fact, we are looking for  Dynamics that are actually conjugate to a transitive subshift of finite type. In order to build such  Dynamics, we will be more precise for the choice of $\cv$  in  sub-section \ref{ssA2}.
\subsection{Rectangles partition}\label{ssA3}
Here we explain how  a good family of rectangles, called a rectangles partition, is useful to build a locally maximal invariant hyperbolic sets. We introduce  geometric Markov partition, that  are reminiscent from the Markov partition and that are studied in \cite{PaTa1987},   but as we didn't find the exact setting that we use elsewhere, we give some details.

We assume that $f:M\rightarrow M$ is a $C^1$ diffeomorphism and that $\cv\subset M$ is an open set endowed with two continuous families of open symmetric cones,  the unstable one $x\in \cv\mapsto C^u(x)\subset T_xM$ and the stable one $x\in \cv\mapsto C^s(x)\subset T_xM$ such that, if we denote the closure of a set $A$ by $\bar A$, we have for a constant  $\lambda\in (0, 1)$ 
\begin{itemize}
\item $\forall x\in \cv\cap f^{-1}(\cv), Df(\overline {C^u}(x))\subset C^u(f(x))\quad {\rm and}\quad Df(C^s(x))\supset \overline{C^s}(f(x))$;
\item $\forall x\in \cv, \forall v \in C^u(x), \| Df(x)v\|\geq \frac{1}{\lambda}\| v\|$ and\\
 $\forall x\in \cv, \forall v \in C^s(x), \| Df(x)v\|\leq {\lambda}\| v\|$;
 \item $\forall x\in \cv, C^u(x)\cap C^s(x)=\{\vec 0\}$.
\end{itemize}
\begin{defi}\begin{itemize}
\item  A $C^1$-embedding $\gamma:[a, b]\rightarrow \cv$ define a {\em unstable (resp. stable) curve} if $\forall t\in [a, b], \gamma'(t)\in C^u(\gamma(t))$ (resp. $\forall t\in [a, b], \gamma'(t)\in C^s(\gamma(t))$).
\item a {\em rectangle} $R$  is given by an embedding $\Phi_R: [0,1]^2\rightarrow R\subset \cv$ such that for every $t\in [0, 1]$, $\Phi_R(\{t\}\times [0, 1])$ (resp. $\Phi_R([0, 1]\times \{t\})$ ) defines a stable (resp. unstable) curve;
\item then the {\em stable (resp. unstable)} boundary of $R$ is $\partial^sR=\Phi_R (\{0, 1\}\times [0, 1])$ (resp. $\partial^uR=\Phi_R([0, 1]\times \{0, 1\})$;
\item a rectangle $R'$ is a {\em stable (resp. unstable) subrectangle} of a rectangle $R$ if $R'\subset R$ and $\partial^uR'\subset \partial^uR$ (resp. $\partial^sR'\subset \partial^sR$).
\end{itemize}

\end{defi}
\begin{remks}
\begin{enumerate}
\item Observe that a stable curve is always transversal to an unstable curve, and that when their mutual intersection with some rectangle is non-empty, then it is a point.
\item to a given rectangle $R$, we can associate different embeddings $\Phi_R$ and then different stable and unstable foliations $\cF^s(R)$ and $\cF^s(R)$;
\item the stable and unstable boundaries are independent from the embedding;
\item\label{pt4} when $\gamma\subset\cv$ is a unstable (stable) curve, every connected component of $f(\gamma)\cap \cv$ (resp. $f^{-1}(\gamma)\cap \cv$) is also an unstable (resp. stable) curve;
\end{enumerate}
\end{remks}
 Let us now introduce the notion of rectangles partition that we will use.

\begin{defi}
A {\em rectangles partition} is a finite set $\{ \cR_1, \dots, \cR_m\}$ of disjoint rectangles of $\cv$ such that, if we use the notation $\cR_{jk}=f(\cR_j)\cap \cR_k$, we have
\begin{itemize}
\item for every $j, k\in\{ 1, \dots, m\}$, either $\cR_{jk}=\emptyset$ or $\cR_{jk}$ is an unstable subrectangle of $\cR _k$. When $\cR_{jk}\not=\emptyset$, we use the notation 
$$\cR_j\xrightarrow[]{f} \cR_k,$$
and we say that we have a {\em transition} from $\cR_j$ to $\cR_k$;
\item when $\cR_{jk}\not=\emptyset$, then $f(\partial^u\cR_j)\cap\partial^u\cR_k=\emptyset$ and $f(\partial^s\cR_j)\cap\partial^s\cR_k=\emptyset$.
\end{itemize}
An {\em admissible sequence} is then $(i_k)_{k\in\Z}\in \{1, \dots , m\}^\Z=\Sigma_m$ such that
$$\forall k\in\Z, \cR_{i_k}\xrightarrow[]{f} \cR_{i_{k+1}}.$$

\end{defi}
\begin{remk}
Observe that $\cR_j\xrightarrow[]{f} \cR_k$ if and only if $\cR_k\xrightarrow[]{f^{-1}} \cR_j$ (the stable boundary for $f^{-1}$ is then the unstable one for $f$).
\end{remk} 

\begin{nota}
We denote by $\Lambda(\cR_1, \dots, \cR_m)$ the maximal invariant set that is contained in $\cR_1\cup\dots \cup \cR_m$, i.e.
$$\Lambda(\cR_1, \dots, \cR_m)=\bigcap_{k\in\Z}f^{k}(\cR_1\cup\dots \cup \cR_m).$$
\end{nota}
Observe that this set is hyperbolic. Hence there exist a stable and an unstable submanifold at every of its points. We even have the following result.
\begin{proposition}\label{Pstable}
If $x\in \Lambda(\cR_1, \dots, \cR_m)\cap \cR_{i_0}$, then the connected component of $W^s(x)\cap \cR_{i_0}$ (resp. $W^u(x)\cap \cR_{i_0}$) that contains $x$ is a stable (resp. unstable) curve that joins the two connected components of $\partial^u\cR_{i_0}$ (resp. $\partial^s\cR_{i_0}$).
\end{proposition}
\begin{proof}
As $\Lambda(\cR_1, \dots, \cR_m)$ is hyperbolic, there exists $\varepsilon>0$ such that for every $x\in\Lambda(\cR_1, \dots, \cR_m)$, the length of  every branch of $W^s(x)$ is greater than $\varepsilon$. We denote by $\cm>0$  a lower bound of the length of the stable curves contained in one $\cR_{j_0}$ that join the two components of $\partial^u \cR_{j_0}$. Then  we choose $N\geq 1$ such that $\frac{\varepsilon}{\lambda^N}>\cm$. Then if $j_0$ is such that  $f^N(x)\in \cR_{j_0}$, the curve  $f^{-N}(W^s(f^N(x))\cap \cR_{j_0})$ is contained in $W^s(x)$ and crosses the two connected components of $\partial^u\cR_{i_0}$. This gives the wanted result.
\end{proof}
Different versions of the following proposition exist in different setting. We will provide a proof for the convenience of the reader.
\begin{proposition}\label{PMarkov}
Let $\{ \cR_1, \dots, \cR_m\}$ be a rectangle partition for $f$ in $\cv$. Let $(i_k)_{k\in \Z}\in\Sigma_m$ be a sequence. The two following assertions are equivalent.
\begin{itemize}
\item $(i_k)_{k\in\Z}$ is an admissible sequence;
\item there exists a unique point $x\in \cR_{i_0}$ such that
$$\forall k\in\Z, f(x)\in \cR_{i_k}.$$
\end{itemize}
\end{proposition}
\begin{proof}
We just prove the direct implication, the only one that is non-trivial. Hence we assume that $(i_k)_{k\in\Z}$ is an admissible sequence. 

We begin by proving the existence of $x$. For every $n\in\N$, we introduce the notation
 $$D_n^s=\bigcap_{k=0}^nf^{-k}(\cR_{i_k})\quad{\rm and}\quad D_n^u=\bigcap_{k=0}^nf^{k}(\cR_{i_{-k}}).$$
 Then $(D_n^u)_{n\in\N}$ (resp. $(D_n^s)_{n\in\N}$) is a decreasing sequence of unstable (resp. stable) rectangles of $\cR_{i_0}$. Hence $(K_n)_{n\in \N}= (D_n^u\cap D_n^s)_{n\in \N}$ is a decreasing sequence of non-empty compact subsets of $\cR_{i_0}$. Their intersection contains at least one point $x$, and this point satisfies
 $$\forall k\in\Z, f^{k}(x)\in \cR_{i_k}.$$
 
 We now want to prove the unicity of $x$. We introduce the notation $$D_\infty^u=\bigcap_{n\in\N} D_n^u\quad{\rm  and} \quad D_\infty^s=\bigcap_{n\in\N}D_n^s.$$
 \begin{lemma}\label{LA1}
 $D_\infty^u$ (resp. $D_\infty^s$) is an unstable curve that joins the two connected components of $\partial^sR_{i_0}$ (resp. $\partial^uR_{i_0}$). More precisely, if $\{ x\}=D_\infty^u\cap D_\infty^s$, then $D_\infty^u\subset W^u(x)$ and  $D_\infty^s\subset W^s(x)$.
 \end{lemma}
 
 Let us prove   Lemma \ref{LA1}. We just prove the result for $D_\infty^s$. As every $D_n^s$ is a stable rectangle, $D_\infty^s$ is a connected compact set that joins the two connected components of $\partial^uR_{i_0}$.\\
  To prove that it is a (at least continuous) curve, we just need to prove that it intersects every leaf of the unstable foliation $\cF^u(\cR_{i_0})$ of $\cR_{i_0}$ at most once. So let $\cl^u$ be an unstable leaf of  $\cR_{i_0}$ and let $x$, $y$ be two points of $D_\infty^s\cap \cl^u$. We denote by $\cl^u[x,y]$  the arc of $\cl^u$ that has for endpoints $x$ and $y$. Observe that $\cl^u[x,y]\subset \cR_{i_0}$ Then for every $n\in\N$,  the connected component $\cl_n$ of $f^{n}(\cl^u)\cap \cR_{i_n}$ that contains $f^{n}(\cl^u[x,y])$\footnote{Observe that the endpoints of this curve are indeed in $\cR_{i_n}$ and hence  by the point (\ref{pt4}) of the remark, $f^{n}(\cl^u[x,y])\subset \cR_{i_n}$.
  } is an unstable curve of $\cR_{i_n}$. Let $\cb$ a common upper bound of the lengths of the unstable leaves that are contained in some rectangle of the Markov partition (observe that these curves are uniformly Lipschitz graphs in the charts $\Phi_{\cR_i}$). Then we have ${\rm length}(\cl_n)\leq \cb$ and we deduce
 $\forall n\in \N, {\rm length}(\cl^u[x,y])={\rm length}(f^{-n}(\cl_n)) \leq \lambda^n\cb$. So $x=y$ and $D_\infty^s$  intersects every unstable leaf at most once, and so exactly once because $D_\infty^s$ is a connected  set that joins the two connected components of $\partial^uR_{i_0}$. 
 
 Moreover, observe that $D_\infty^s$  contains the connected component $C^s$ of $W^s(x)\cap \cR_{i_0}$ that contains $x$. This implies that $D_\infty^s=C^s$ is a smooth stable curve (see Proposition \ref{Pstable}).

\end{proof}

\subsection{Precise construction of heteroclinic horseshoes}\label{ssA2} We use the same notations as in subsection \ref{ssA1}.

\begin{remk}
As explained before, we want to build an invariant set that is close (for the Hausdorff distance) to $K(\ch)$. That is why we need to use all the heteroclinic intersections that are in $K(\ch)$ in our construction. Another approach could be to use the transitivity of the relation $\cR$ defined on $q$-periodic points by: $x\cR y$ if $W^s(x,f)$ and $W^u(y,f)$ have a transverse heteroclinic intersection.  This implies that every periodic point in $K(\ch)$ has a homoclinic intersection and thus we could use directly Smale's method (see \cite{Sma1965}) to build a homoclinic horseshoe. Unfortunately, a neighbourhood of this homoclinic orbit is not necessarily a neighbourhood of the whole $K(\ch)$ and so this horseshoe is in general not close to $K(\ch)$ for the Hausdorff distance, so doesn't give us what we want.
\end{remk}

\begin{theorem}
 There exists $N\geq 1$ and a neighbourhood $\cN$ of the cycle $K(\ch)$ of transverse heteroclinic intersections with  period $q$, such that the maximal $f^{qN}$ invariant set contained in $\cN$ is a horseshoe $\Lambda$ for $f^{qN}$ (see Definition \ref{Defhs}).
 \end{theorem}

As $K(\ch)$ is (uniformly) hyperbolic, we can chose a neighbourhood $\cv$ of $K(\ch)$, a constant $\lambda\in (0, 1)$ and two continuous families of open symmetric cones (see subsection \ref{ssA1}) the unstable one $x\in \cv\mapsto C^u(x)\subset T_xM$ and the stable one $x\in \cv\mapsto C^s(x)\subset T_xM$ such that, if we denote the closure of a set $A$ by $\bar A$, we have
\begin{itemize}
\item $\forall x\in \cv\cap f^{-1}(\cv), Df(\overline {C^u}(x))\subset C^u(f(x))\quad {\rm and}\quad Df(C^s(x))\supset \overline{C^s}(f(x))$;
\item $\forall x\in \cv, \forall v \in C^u(x), \| Df(x)v\|\geq \frac{1}{\lambda}\| v\|$ and\\
 $\forall x\in \cv, \forall v \in C^s(x), \| Df(x)v\|\leq {\lambda}\| v\|$;
 \item $\forall x\in \cv, C^u(x)\cap C^s(x)=\{\vec 0\}$.
\end{itemize}


\begin{nota}
For every $x_k$, we denote by $B^s(x_k)$ the branch of $W^s(x_k)$ that contains the $y^{k-1}_i$s and by $B^u(x_k)$ the branch of $W^u(x_k)$ that contains the $y^{k}_i$s.\\
Then we choose a small (curved) rectangle $R_k$ with two sides on $B^s(x_k)$ and $B^u(x_k)$. 

\begin{center}
\includegraphics[width=10cm]{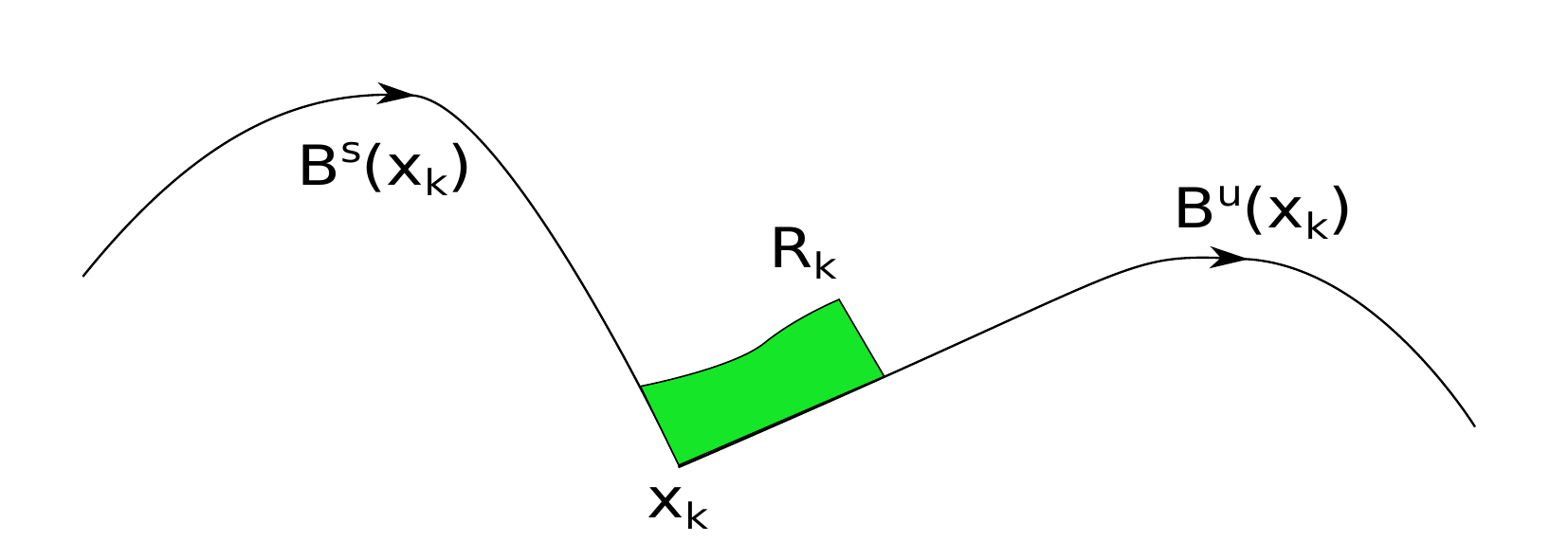}
\end{center}



We denote by $\delta_k^u$ and $\delta_k^s$ the size of $R_k$ along $B^u(x_k)$ and $B^s(x_k)$.  
\end{nota}
Then we look at the Poincar\'e map for $f^q$ from $R_k$ onto $R_{k+1}$. Adjusting the quantities $\delta^u$ and $\delta^s$, we can find some $N_k$ such that $f^{qN_k}(R_k)\cap R_{k+1}$ contains the union of a finite numbers of unstable rectangles. There are two cases.
\begin{itemize}
\item when $n=q=1$, there are $n_0+1$ rectangles: $R_0^0$ that contains $x_0$ and  $R_{0}^1$, $R_{0}^2$, \dots, $R_{0}^{n_{0}}$ such that $R_{0}^i$ is a connected component of $f^{qN_0}(R_0)\cap R_{0}$ that meets $W^s_{\rm loc}(x_{0})$ at some point of the orbit of $y_i^0$;
\item when $nq>1$, there are $n_{k}$ unstable sub-rectangles of $R_{k+1}$ that we denote by $R_{k+1}^1$, $R_{k+1}^2$, \dots, $R_{k+1}^{n_{k}}$ such that $R_{k+1}^i$ is a connected component of $f^{qN_k}(R_k)\cap R_{k+1}$ that meets $W^s_{\rm loc}(x_{k+1})$ at some point of the orbit of $y_i^k$.\footnote{Observe that $f^{qN_k}(R_k)\cap R_{k+1}$ can have other connected components, for example connected components that correspond to other heteroclinic intersections. We just work with some chosen heteroclinic points}
\end{itemize} 
\begin{center}
\includegraphics[width=10cm]{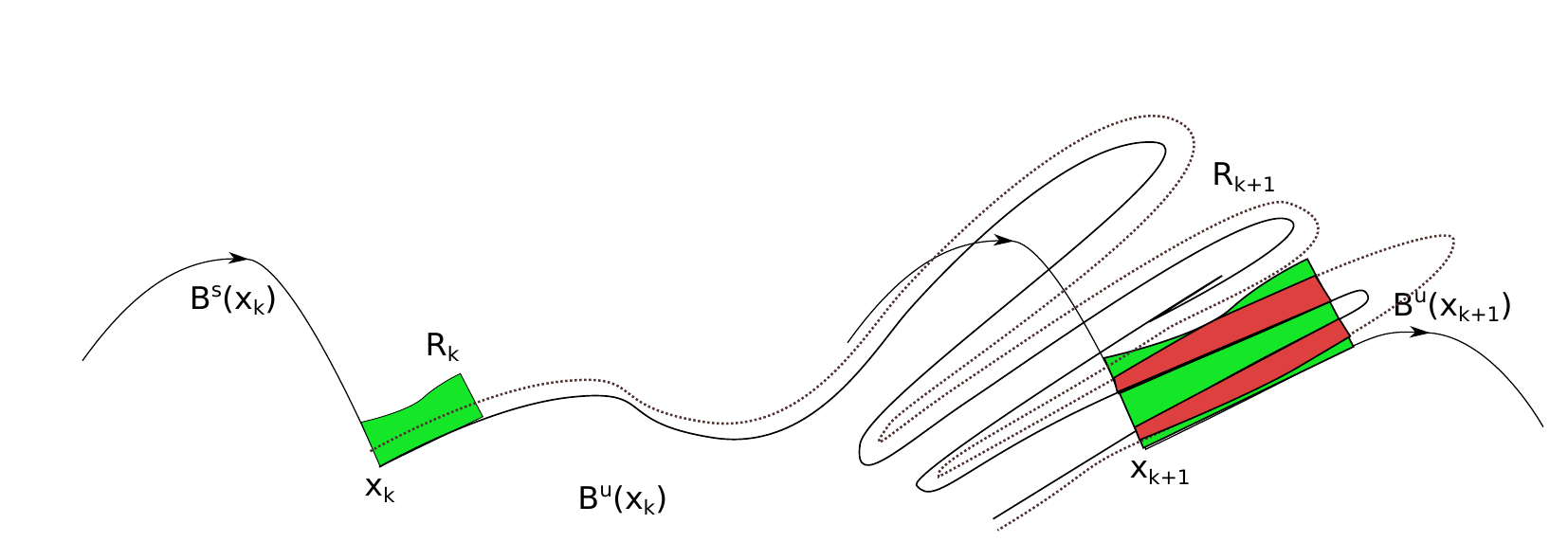}
\end{center}
When we decrease $\delta_k^u$ or $\delta_{k+1}^s$, then $N_k$ increases and when we decrease $\delta_k^s$ or $\delta_{k+1}^u$, then $N_k$ doesn't change.   Hence, if we eventually decrease the $\delta_k^u$'s, we can assume that all the $N_k$ are equal to some constant integer that we denote by $N$. 
Let us denote by $R_k^0$ the connected component of $R_k\cap f^{qN}(R_k)$ that contains $x_k$  and let us prove that it is disjoint from the $R_k^i$ for $1\leq i\leq n_k$. There are two cases.
\begin{itemize}
\item there is only one fixed point in the heteroclinic cycle, i.e. $q=n=1$; in this case the rectangles $R_1^i$ are different connected components of $R_1\cap f^N(R_1)$ and so they are disjoint;
\item if not, as the different $R_k$'s are disjoint, in particular $f^{qN}(R_k)$ and $f^{qN}(R_{k-1})$ are disjoint and every unstable rectangle that is contained in $R_k\cap f^{qN}(R_k)$ is disjoint from $\displaystyle{\bigcup_{i=1}^{n_{k-1}} R_k^i}$.
\end{itemize}
\begin{center}
\includegraphics[width=10cm]{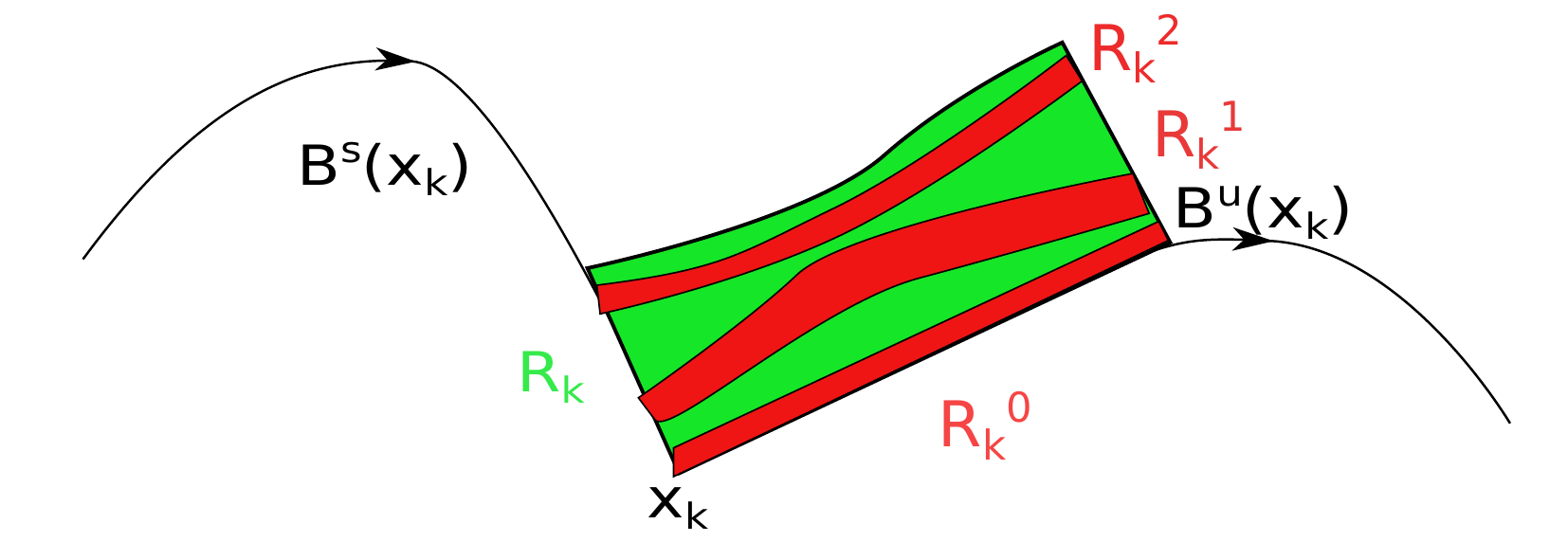}
\end{center}
We introduce the notation $\displaystyle{\ct_k=\bigcup_{i=0}^{n_k}R_k^i}$ and consider now the $f^{qN}$-invariant set
$$\Lambda=\bigcap_{j\in\Z} f^{jqN}(\bigcup_{k=1}^{qn}\ct_k).$$
Then the $R_k^j$s with $1\leq k\leq nq$ and $0\leq j\leq n_k$ define a rectangle partition   for $f^{qN}_{|\cv}$ and the following  transitions occur\footnote{We don't know if other transitions occur.}\begin{itemize}
\item $\forall i\in [0, n_k],R_k^i\xrightarrow[]{f^{qN}} R_k^0$; 
\item $\forall i\in [0, n_k],\forall j\in [1, n_{k+1}],  R_k^i\xrightarrow[]{f^{qN}} R_{k+1}^j$;
\end{itemize}
We denote by $A$ the associated matrix. Observe that for every $R_k^i$, $R_h^j$, then $R_k^i$ can be connected to  $R_h^j$ by a succession of such transitions. We deduce from Proposition \ref{PMarkov} that $f^{Nq}_{|\Lambda}$ is conjugate to the subshift associated to $A$, that is transitive. In particular, $f^{Nq}_{|\Lambda}$ is mixing, has an infinity of periodic points and has positive topological entropy.  

\begin{remks}
\begin{itemize}
\item If we decrease the constants $\delta_k^u$ and $\delta_k^s$, then we increase $N$ but this is not a problem because we just add some iterations of $f^q$ that are close to the periodic orbits where we know exactly how the Dynamics looks like. An advantage is that decreasing sufficiently these constants, we can be sure that $\displaystyle{\bigcup_{j=0}^{qN}       f^j\left(  \bigcup_{k=1}^{qn}\ct_k\right)}$ is contained in a small neighbourhood of the heteroclinic cycle $K(\ch)$. So in this case, the Hausdorff distance between $K(\ch)$ and the invariant set $\displaystyle{\bigcup_{j=1}^{qN}f^j(\Lambda)}$ is also as small as we want;
\item being defined by a rectangle partition, the set $\Lambda$ is a locally maximal invariant set by $f^{qN}$.
\end{itemize}
\end{remks}

\bibliographystyle{amsplain}

\end{document}